\documentclass[leqno,11pt]{amsart}
\usepackage{}
\usepackage{amssymb, amsmath}
\usepackage{amsthm, amsfonts,mathrsfs}
\def\inte#1{
\displaystyle\mathop{#1\kern0pt}^\circ }

\let\f=\frac

\let\wt=\widetilde

\def\cS{{\mathcal S}}

\def\virgp{\raise 2pt\hbox{,}}
\def\cdotpv{\raise 2pt\hbox{;}}

\def\C{\mathop{\mathbb C\kern 0pt}\nolimits}
\def\DD{\mathop{\mathbb D\kern 0pt}\nolimits}
\def\EE{\mathop{{\mathbb E \kern 0pt}}\nolimits}
\def\K{\mathop{\mathbb K\kern 0pt}\nolimits}
\def\N{\mathop{\mathbb N\kern 0pt}\nolimits}
\def\Q{\mathop{\mathbb Q\kern 0pt}\nolimits}
\def\R{\mathop{\mathbb R\kern 0pt}\nolimits}
\def\SS{\mathop{\mathbb S\kern 0pt}\nolimits}
\def\ZZ{\mathop{\mathbb Z\kern 0pt}\nolimits}
\def\TT{\mathop{\mathbb T\kern 0pt}\nolimits}
\def\P{\mathop{\mathbb P\kern 0pt}\nolimits}

\newcommand{\la}{\lambda}

\newcommand{\beq}{\begin{equation}}
\newcommand{\eeq}{\end{equation}}
\newcommand{\ben}{\begin{eqnarray}}
\newcommand{\een}{\end{eqnarray}}
\newcommand{\beno}{\begin{eqnarray*}}
\newcommand{\eeno}{\end{eqnarray*}}
\newtheorem{defi}{Definition}[section]
\newtheorem{thm}{Theorem}[section]
\newtheorem{lem}{Lemma}[section]
\newtheorem{rmk}{Remark}[section]
\newtheorem{col}{Corollary}[section]
\newtheorem{prop}{Proposition}[section]
\renewcommand{\theequation}{\thesection.\arabic{equation}}

\begin{document}

\title[Global smooth solutions to the 3D non-resistive MHD equations]
{Global smooth solutions to the 3D non-resistive MHD equations with low regularity axisymmetric data}

\author[X. Ai and Z. Li]{Xiaolian Ai\textsuperscript{1}      \and   Zhouyu Li\textsuperscript{2,*}}

%\thanks{Received date 2021.02.14, and accepted date 2022.02.07}
\thanks{$^*$ Corresponding author}
\thanks{$^1$ School of Mathematics, Northwest University, Xi'an 710069, China}
\thanks{$^2$ School of Sciences, Xi'an University of Technology, Xi'an 710054, China}
\thanks{E-mail address: aixl@nwu.edu.cn (X. Ai); zylimath@163.com (Z. Li)}

\begin{abstract}
The purpose of this paper is to study the incompressible non-resistive MHD equations in $\mathbb{R}^3$.
We establish the global well-posedness of the system if the initial data is axially symmetric and the swirl component
of the velocity and the magnetic vorticity vanish. In particular, the special axially symmetric initial data can be arbitrarily large and satisfy low regularity assumptions.

\end{abstract}

%%%%%%%%%%%%%%

\date{}

\maketitle

%%%%%%%%%%%%%%%%%%%%%%%
%\tableofcontents


\vskip 0.2cm

%%%%%%%%%%%%%%%%%%%%%%%%%%%%%%%%%%%%%%%%%%%%%%%%%%%%%%%%%%%%
\renewcommand{\theequation}{\thesection.\arabic{equation}}
\setcounter{equation}{0}
%%%%%%%%%%%%%%%%%%%%%%%%%%%%%%%%%%%%%%%%%%%%%%%%%%%%%%%%%%%%
\section{Introduction}

The magneto-hydrodynamics (MHD) equations were first introduced by Hannes Alfv$\operatorname{\acute{e}}$n \cite{Alf1942}, for which he won the Nobel Prize in Physics in 1970. It is a combination of Navier-Stokes equations of fluid dynamics and Maxwell's equations of electromagnetic field, and describes the time evolution of electrically-conducting fluids. The three dimensional incompressible MHD equations are described as follows:
\begin{equation}\label{VMHD}
    \begin{cases}
    \partial_t u + u\cdot \nabla u + \nabla P  =\nu\Delta u + B\cdot \nabla B,\\
    \partial_t B + u\cdot \nabla B  =  \eta\Delta B+ B\cdot \nabla u,\\
    \nabla\cdot u=\nabla\cdot B = 0.
    \end{cases}
\end{equation}
Here $x\in \mathbb{R}^3$ is the spatial coordinates and $t\geq 0$ is time.
%The unknown functions
$u$, $B$ and $P$ denote the velocity of the fluid, the magnetic field and the pressure, respectively.
The coefficients $\nu$ and $\eta$ are nonnegative constants.
If $\nu>0$ and $\eta=0$, we say MHD system is non-resistive. Without loss of generality, we take $\nu=1$ and then the system \eqref{VMHD} becomes
\begin{equation}\label{MHD}
    \begin{cases}
    \partial_t u + u\cdot \nabla u + \nabla P = \Delta u + B\cdot \nabla B,\\
    \partial_t B + u\cdot \nabla B = B\cdot \nabla u,\\
    \nabla\cdot u=\nabla\cdot B = 0.
    \end{cases}
\end{equation}
The MHD system is widely used in the study of astrophysics,
geophysics and cosmology. For more physical explanations, see \cite{Dav2001, Gui2014, Li2021, Pri2000}.

Before proceeding, we first introduce a vector field $f$ which is axisymmrtric, this means that it has the form
$$f(t, x)=f^r(t, r, z)e_r + f^\theta(t, r, z)e_\theta + f^z(t, r, z)e_z.$$
Here $(r, \theta, z)$ is the cylindrical coordinate system, that is, for any $x=(x_1, x_2, x_3)\in \mathbb{R}^3$,
\begin{equation*}
     \begin{split}
   r=\sqrt{x_1^2+x_2^2},\qquad \theta=\arctan\frac{x_2}{x_1}, \qquad  z=x_3.
     \end{split}
\end{equation*}
$(e_r, e_\theta, e_z)$ is the cylindrical basis in $\mathbb{R}^3$, which is defined by
\begin{equation*}
     \begin{split}
   e_r=(\frac{x_1}{r}, \frac{x_2}{r}, 0), \qquad e_\theta=(-\frac{x_2}{r}, \frac{x_1}{r}, 0), \qquad e_z=(0, 0, 1).
     \end{split}
\end{equation*}
We say $f^\theta$ is the swirl component and $f$ is axisymmetric without swirl if $f^\theta =0$.

Note that if $B=0$, the MHD system reduces to the classical incompressible Navier-Stokes equations.
It is well-known that the global well-posedness of the 3D Navier-Stokes equations is still unsolved. Thus many
efforts try to study the solutions with some special structures. For example, assuming that the initial data
is axisymmetric without swirl, Ukhovskii and Yudovich \cite{U1968} and Ladyzhenskaya \cite{L1968} independently
showed weak solutions are regular. More precisely, the Navier-Stokes equations have a unique global solution for
$u_0\in H^s(\mathbb{R}^3)$, $s>\frac{7}{2}$. Moreover, Leonardi et al. in \cite{L1999} weakened the initial condition to
$u_0\in H^2(\mathbb{R}^3)$ and Abidi in \cite{Abidi2008} proved the global well-posedness for $u_0\in H^\frac{1}{2}(\mathbb{R}^3)$.
For the case axisymmetric with non-trivial swirl, the results need to take some smallness assumptions on initial data. The interested readers may refer to \cite{Zhang2018, Zhang2014}.

For the MHD system, there are also lots of important results up to date. For the case of \eqref{VMHD},
Duvaut and Lions \cite{Duv1972} proved the local well-posedness in Sobolev space $H^s(\mathbb{R}^n)$, $s\geq n$
and Sermange and Temam \cite{Ser1983} showed the global well-posedness in the 2D case. For the case of \eqref{MHD},
Fefferman er al. in \cite{Feff2014} established the local existence and uniqueness of solutions with initial data $(u_0, B_0)\in H^s(\mathbb{R}^n)$, $s > \frac{n}{2}$ $(n=2,3)$. In \cite{Feff2017}, the initial data regularity was weakened to $(u_0, B_0)\in H^{s-1+\varepsilon}(\mathbb{R}^n)\times H^s(\mathbb{R}^n)$,  $s> \frac{n}{2}$ $(n=2,3)$ and $0<\varepsilon<1$. For a class of axisymmetric initial data, Lei \cite{Lei2015} investigated the global well-posedness of the system \eqref{MHD} with a specific geometrical assumption. More precisely, under
the assumptions that swirl component of velocity and magnetic vorticity are trivial, he proved
that there exists a unique global solution with initial data
\begin{equation}\label{W-Condition}
     \begin{split}
   (u_0, B_0)\in H^s(\mathbb{R}^3),\,\,\,  s\geq 2, \qquad \mbox{and} \qquad \frac{B^\theta_0}{r}\in L^\infty(\mathbb{R}^3).
     \end{split}
\end{equation}
Later on, Liu \cite{Liu2018} further obtained the global well-posedness of the system \eqref{VMHD} in the case where the
swirl component of velocity is non-trivial.

Motivated by Lei \cite{Lei2015}, we are concerned with the global well-posedness of the axisymmetric MHD system \eqref{MHD}. It should be pointed that for $s\geq 2$, $B_0 \in H^s(\mathbb{R}^3)$ can not derive $\frac{B^\theta_0}{r}\in L^\infty(\mathbb{R}^3)$ by Sobolev embedding $H^m(\mathbb{R}^3)\hookrightarrow L^\infty(\mathbb{R}^3)$ $(m>\frac{3}{2})$. Thus, a natural and interesting problem is whether or not the assumption conditions \eqref{W-Condition} can be weakened. In the present paper, we give a positive answer. The main result of this work reads as follows.

%In this paper, we consider the system \eqref{MHD} in $\mathbb{R}^3$ is globally well-posed for a class of special axiymmetric initial data whose swirl components of the velocity field and magnetic vorticity field vanish, that is
%\begin{equation*}
%     \begin{split}
%   &u_0=u^r_0(r, z)e^r + u^z_0(r, z)e^z,\\
%   &B_0=B^\theta_0(r, z)e^\theta.
%     \end{split}
%\end{equation*}

\begin{thm}\label{thm-main}
Suppose that $u_0$ and $B_0$ are both axially symmetric divergence free vector fields such that $u_0^\theta=B_0^r=B_0^z=0$.
Let $(u_0, B_0)\in H^{1}{(\mathbb{R}^3)}\times H^2{(\mathbb{R}^3)}$, and $\frac{\omega_0}{r}\in L^2(\mathbb{R}^3)$. Then there exists a unique global solution $(u, B)$ to the system \eqref{MHD} satisfying
\begin{equation*}
     \begin{split}
&u\in L^\infty(0, T; H^1(\mathbb{R}^3))\cap L^1(0, T; W^{1, \infty}(\mathbb{R}^3)),\\
&B\in L^\infty(0, T; H^2(\mathbb{R}^3)),\quad \frac{\omega}{r}\in L^\infty(0, T; L^2(\mathbb{R}^3)),
     \end{split}
\end{equation*}
for any $0< T< \infty$.
%$H^s(\mathbb{R}^3)$ norm of
\end{thm}

\begin{rmk}
{\rm{(i).}} Taking advantage of the estimate of $u$ in $W^{1, \infty}$, which is given by Theorem \ref{thm-main},
it is not hard to propagate by classical arguments higher order regularity, for example higher $H^s$ Sobolev
regularity.\\
{\rm{(ii).}} Compared to the result in \cite{Lei2015}, we emphasize two points.
%in Theorem \ref{thm-main}.
The first one is to remove the condition $\frac{B_0^\theta}{r}\in L^\infty(\mathbb{R}^3)$.
The other is to weaken the condition $u_0\in H^2(\mathbb{R}^3)$. In fact,
we know that in cylindrical coordinates the vorticity of the swirl-free axisymmetric velocity
is given by
\begin{equation*}
  \begin{split}
\omega= \nabla \times u = \omega^\theta e_\theta
 \end{split}
\end{equation*}
with $\omega^\theta=\partial_z u^r-\partial_r u^z$, and
\begin{equation*}
  \begin{split}
|\nabla^2 u| \sim |\nabla \omega^\theta| + |\frac{\omega^\theta}{r}|.
 \end{split}
\end{equation*}
%Thus, it is not hard to find that our result can be regarded as a relaxation to the result of Lei in \cite{Lei2015}.\\
%with respect to the assumptions of the initial data velocity and magnetic field.\\
{\rm{(iii).}} For the fully viscid MHD system \eqref{VMHD}, we also can obtain similar result of Theorem \ref{thm-main}.
%We leave this to the interested readers.
\end{rmk}

The proof of Theorem \ref{thm-main} strongly relies on the structure of the MHD equations in axially symmetric case
whose the swirl component of velocity and magnetic vorticity vanish.
In contrast with the proof in \cite{Lei2015},
due to the absence of the conditions $ u_0 \in H^2(\mathbb{R}^3)$  and $\frac{B^\theta_0}{r}\in L^\infty(\mathbb{R}^3)$,
we have to estimate more carefully to obtain $u\in L^1([0, T]; \operatorname{Lip}(\mathbb{R}^3))$. On the other hand, we can not
derive the $L^1([0, T]; \operatorname{Lip}(\mathbb{R}^3))$ estimate for $B$, which plays the key role in the proof of \cite{Lei2015}. Hence, in order to obtain the $H^2$ estimate of $B$, we show the two estimates
$\|\nabla (\frac{B^\theta}{r})\|_{L^\infty([0, T]; L^2(\mathbb{R}^3))}$ and $\|\nabla u\|_{L^1([0, T]; H^{2-\epsilon}(\mathbb{R}^3))}$ with $0<\epsilon<1$, for more details see Proposition \ref{prop-7a} and Proposition \ref{prop-7b} below.

The paper is organized as follows. In Section 2, we introduce the system \eqref{MHD} in cylindrical coordinates,
recall the definition of Besov spaces and gather some elementary facts.
In Section 3, we give some a priori estimates and then prove Theorem \ref{thm-main}.

\medbreak \noindent{\bf Notations:} We shall denote
$\int\cdot dx\triangleq\int_{\mathbb{R}^3}\cdot dx$ and use the letter $C$ to denote a generic constant, which may vary from line to line.
For a Banach space $B$, sometimes we use the notation $L^p_t B$ for $L^p([0, t]; B)$.
We always use $X\lesssim Y$ to denote $X\leq CY$. Finally, $X\sim Y$ stands for $X\lesssim Y$ and $X\gtrsim Y$.

%%%%%%%%%%%%%%%%%%%%%%%%%%%%%%%%%%%%%%%%%%%%%%%%%%%%%%%%%%%

\renewcommand{\theequation}{\thesection.\arabic{equation}}
\setcounter{equation}{0} %%%%%%%%%%%%%%%%%%%%%%%%%%

%%%%%%%%%%%%%%%%%%%%%%%%%%%%%%%%%%%%%%%%%%%%%%%%%%%%%%%%%%%
\section{Preliminaries}
In this section, we will introduce the system \eqref{MHD} in cylindrical coordinates, Besov spaces and some useful inequalities.
Considering the system \eqref{MHD} in the cylindrical coordinates, we can write
\begin{equation*}
     \begin{split}
   &u(t, x)=u^r(t, r, z)e_r + u^\theta(t, r, z)e_\theta + u^z(t, r, z)e_z,\\
   &B(t, x)=B^r(t, r, z)e_r + B^\theta(t, r, z)e_\theta + B^z(t, r, z)e_z,\\
   &P(t,x)=P(t, r, z).
     \end{split}
\end{equation*}
Then the system \eqref{MHD} can be equivalently reformulated as
\begin{equation}\label{MHD-1}
    \begin{cases}
    \partial_t u^r + u^r\partial_r u^r+ u^z\partial_z u^r-\frac{(u^\theta)^2}{r}+ \partial_r P \\
    \qquad\quad= (\Delta-\frac{1}{r^2})u^r+ B^r\partial_r B^r + B^z\partial_z B^r -\frac{(B^\theta)^2}{r},\\
    \partial_t u^\theta + u^r\partial_r u^\theta+ u^z\partial_z u^\theta+\frac{u^ru^\theta}{r} \\
    \qquad\quad= (\Delta-\frac{1}{r^2})u^\theta+ B^r\partial_r B^\theta + B^z\partial_z B^\theta +\frac{B^r B^\theta}{r},\\
    \partial_t u^z + u^r\partial_r u^z+ u^z\partial_z u^z+ \partial_z P = \Delta u^z +B^r\partial_r B^z +B^z\partial_zB^z,\\
    \partial_t B^r + u^r\partial_r B^r+ u^z\partial_z B^r = B^r \partial_r u^r + B^z \partial_z u^r,\\
    \partial_t B^\theta + u^r\partial_r B^\theta + u^z\partial_z B^\theta + \frac{B^r u^\theta}{r} = B^r\partial_r u^\theta+
    B^z\partial_z u^\theta +\frac{u^r B^\theta}{r},\\
    \partial_t B^z + u^r\partial_r B^z+ u^z\partial_z B^z = B^r \partial_r u^z + B^z \partial_z u^z,\\
    \partial_r u^r+\frac{u^r}{r}+\partial_z u^z =0, \quad \partial_r B^r+\frac{B^r}{r}+\partial_z B^z =0,
    \end{cases}
\end{equation}
where
%the pressure $P=p+\frac{1}{2}|B|^2$ and
$\Delta=\frac{\partial^2}{\partial r^2}+\frac{1}{r}\frac{\partial}{\partial r}+\frac{\partial^2}{\partial z^2}$ is the Laplacian operator.

Taking advantage of the local existence and uniqueness result for the system \eqref{MHD} in $\mathbb{R}^3$,
we can obtain the following lemma.
\begin{lem}[\cite{Feff2017}] \label{lem-0}
Let $(u_0, B_0)\in H^{1}{(\mathbb{R}^3)}\times H^2{(\mathbb{R}^3)}$, and $\frac{\omega_0}{r}\in L^2(\mathbb{R}^3)$ be axially symmetric divergence-free vector fields. Then there exists $T>0$ and a unique axially symmetric solution $(u, B)$ on $[0, T)$ for the system \eqref{MHD-1} such that
\begin{equation*}
     \begin{split}
&u\in L^\infty(0, T; H^1(\mathbb{R}^3)),\quad B\in L^\infty(0, T; H^2(\mathbb{R}^3)),\\
&\frac{\omega}{r}\in L^\infty(0, T; L^2(\mathbb{R}^3)).
     \end{split}
\end{equation*}
\end{lem}
By using the uniqueness of local solutions, it is not difficult to find that if $u_0^\theta=B_0^r=B_0^z=0$, then
$u^\theta=B^r=B^z=0$ for all later times. In this case, the system \eqref{MHD} can be simplified as
\begin{equation}\label{MHD-2}
    \begin{cases}
    \partial_t u^r + u^r\partial_r u^r+ u^z\partial_z u^r+ \partial_r P = (\Delta-\frac{1}{r^2})u^r-\frac{(B^\theta)^2}{r},\\
    \partial_t u^z + u^r\partial_r u^z+ u^z\partial_z u^z+ \partial_z P = \Delta u^z,\\
    \partial_t B^\theta + u^r\partial_r B^\theta+ u^z\partial_z B^\theta = \frac{u^r B^\theta}{r},\\
    \partial_r u^r+\frac{u^r}{r}+\partial_z u^z =0.
    \end{cases}
\end{equation}

Let us define
\begin{equation*}\label{1}
     \begin{split}
\Pi :=\frac{B^\theta}{r}, \quad \Omega :=\frac{\omega^\theta}{r}.
     \end{split}
\end{equation*}
The system \eqref{MHD-2} gives
\begin{equation}\label{MHD-3}
    \begin{cases}
    \partial_t \Pi + u\cdot \nabla \Pi = 0,\\
    \partial_t \Omega + u\cdot \nabla \Omega = (\Delta + \frac{2}{r}\partial_r)\Omega -\partial_z \Pi^2,
    \end{cases}
\end{equation}
where $u\cdot \nabla f= u^r \partial_r f +u^z\partial_z f$ for $f=f(t, r, z)$.

To the end, we give the definition of Besov spaces and some useful inequalities. Let us first recall the classical dyadic decomposition in $\mathbb{R}^3$, see \cite{Che1998}.
Let $\varphi$  and $\chi$ be two smooth functions supported in the ring
$\mathcal{C}:=\{\xi\in \mathbb{R}^3, \frac{3}{4}\leq|\xi|\leq \frac{8}{3}\}$
and the ball $\mathcal{B}:=\{\xi\in \mathbb{R}^3, |\xi|\leq \frac{3}{4}\}$ respectively
such that
\begin{equation*}
    \begin{split}
   \sum_{j\in \mathbb{Z}}\varphi (2^{-j}\xi)=1 \quad \mbox{for} \quad \xi\neq 0
   \quad  \mbox{and} \quad \chi(\xi)+\sum_{q\geq 0}\varphi (2^{-q}\xi)= 1 \quad \forall\, \xi\in \mathbb{R}^3.
    \end{split}
\end{equation*}
For every $u\in \mathcal{S}'(\mathbb{R}^3)$, we set
\begin{equation*}
    \begin{split}
   &\forall\, q\in \mathbb{Z}, \quad \dot{\Delta}_q u= \varphi (2^{-q}\mathcal{D})u, \quad \mbox{and} \quad  \dot{S}_q u= \sum_{ j\leq q-1} \Delta_j u.\\
   &q\geq 0, \quad \Delta_q u= \varphi (2^{-q}\mathcal{D})u, \quad \Delta_{-1} u=\chi(\mathcal{D})u
   \quad \mbox{and} \quad S_q u= \sum_{-1\leq q'\leq q-1} \Delta_{q'} u.
    \end{split}
\end{equation*}
Then we get the decomposition
\begin{equation*}
    \begin{split}
   u=\sum_{q\in \mathbb{Z}}\dot{\Delta}_q u,\quad \forall \, u\in \mathcal{S}'(\mathbb{R}^3)/ \mathcal{P}[\mathbb{R}^3] \quad \mbox{and} \quad
   u=\sum_{q\geq -1}\Delta_q u,\quad \forall \, u\in \mathcal{S}'(\mathbb{R}^3),
    \end{split}
\end{equation*}
where $ \mathcal{P}[\mathbb{R}^3]$ is the set of polynomials, see \cite{Che1998}.
Let us recall the definition of nonhomogeneous and homogeneous Besov spaces.
\begin{defi}
Let $(p, r)\in [1, +\infty]^2$, $s\in \mathbb{R}$ and $u\in \mathcal{S}'(\mathbb{R}^3)$,  we set
\begin{equation*}
    \begin{split}
   \|u\|_{B^s_{p, r}}:= \Big(2^{qs} \|\Delta_q u\|_{L^p}\Big)_{\ell^r} \quad \mbox{and} \quad
   \|u\|_{\dot{B}^s_{p, r}}:= \Big(2^{qs} \|\dot{\Delta}_q u\|_{L^p}\Big)_{\ell^r},
    \end{split}
\end{equation*}
with the usual modification if $r=\infty.$
\begin{itemize}
\item
For $s\in\R$, we define
$B^s_{p, r}(\R^3):= \big\{ u\in{\mathcal S}'(\R^3) \;\big|\; \|u\|_{B^s_{p, r}}<\infty\big\}.$

%\item
%We define
%$\dot B^s_{p, r}(\R^3):= \big\{ u\in{\mathcal S}'_h(\R^3) \;\big|\; \|u\|_{\dot{B}^s_{p, r}}<\infty\big\}.$
%Here ${\mathcal S}'_h$ is the space of temperate distributions $u$ such that $\displaystyle\lim_{j\rightarrow -\infty} \dot{S}_j u= 0.$
\item
For $s<\f3p$ (or $s=\f3p$ if $r=1$), we define
$\dot B^s_{p,\,r}(\R^3):= \big\{ u\in{\mathcal S}'(\R^3) \;\big|\; \|u\|_{\dot{B}^s_{p, r}}<\infty\big\}.$

\item
If $k\in \N $ and $\frac{3}{p}+k \leq s< \frac{3}{p}+k+1 $(or $s=\frac{3}{p}+k+1$ if $r=1$), then $\dot B^s_{p,\,r}(\R^3)$ is defined as the subset of distributions $u \in {\mathcal S}'(\R^3)$ such that
$\partial^\beta u\in \dot B^{s-k}_{p,\,r}(\R^3)$ whenever $|\beta|=k$.
\end{itemize}
\end{defi}
\begin{rmk}
It should be noted that the homogeneous Besov space $\dot{B}^s_{2, 2}(\R^3)$ (resp. $B^s_{2, 2}(\R^3)$) coincides with
the homogeneous Sobolev space $\dot{H}^s(\R^3)$ (resp. $H^s(\R^3)$).
\end{rmk}

Next, we recall the Bernstein inequalities.
\begin{lem}[\cite{Che1998}]\label{lem-Bern}
Let $\mathcal{B}$ be a ball and $\mathcal{C}$ be a ring of $\mathbb{R}^3$. There exists a constant $C$ such
that for any positive number $\delta$, any non-negative integer $k$, any smooth homogeneous function $\sigma$ of degree
$m$, and any couple of real numbers $(a, b)$ with $b\geq a\geq 1$, one has
\begin{equation*}
    \begin{split}
   &\operatorname{Supp} \hat{u}\subset \delta \mathcal{B} \Rightarrow \sup_{|\alpha|=k}\|\partial^\alpha u\|_{L^b}
   \leq C^{k+1} \delta^{k+3(\frac{1}{a}-\frac{1}{b})} \|u\|_{L^a}, \\
   &\operatorname{Supp} \hat{u}\subset \delta \mathcal{C} \Rightarrow C^{-1-k} \delta^{k}\|u\|_{L^a}
   \leq \sup_{|\alpha|=k} \|\partial^\alpha u\|_{L^a}\leq C^{1+k} \delta^{k}\|u\|_{L^a}, \\
   &\operatorname{Supp} \hat{u}\subset \delta \mathcal{C} \Rightarrow \
   \|\sigma(\mathcal{D}) u\|_{L^b}\leq C_{\sigma, m}\delta^{m+3(\frac{1}{a}-\frac{1}{b})}\|u\|_{L^a}.
    \end{split}
\end{equation*}
\end{lem}

By using the Bernstein inequalities, we have the following continuous embedding:
\begin{equation*}
    \begin{split}
  B^s_{p_1, r_1}(\mathbb{R}^3)\hookrightarrow B^{s+ 3(\frac{1}{p_2}-\frac{1}{p_1})}_{p_2, r_2}(\mathbb{R}^3)
    \end{split}
\end{equation*}
with $p_1\leq p_2$ and $r_1\leq r_2$.

The so-called tame estimate will be stated as follows.
\begin{lem}[\cite{Che1998}]\label{lem-Besov}
For any $s>0$ and $1\leq p, r \leq\infty$, there exists a constant $C>0$ such
that
\begin{equation*}
    \begin{split}
   \|fg\|_{B^s_{p, r}(\mathbb{R}^3)}
   \leq \frac{C^{s+1}}{s}\left( \|f\|_{L^{\infty}(\mathbb{R}^3)}\|g\|_{B^s_{p, r}(\mathbb{R}^3)} + \|g\|_{L^{\infty}(\mathbb{R}^3)}\|f\|_{B^s_{p, r}(\mathbb{R}^3)} \right).
    \end{split}
\end{equation*}
\end{lem}

We also need the following commutator estimate.
\begin{lem}[\cite{Ken1991}]\label{lem-comm}Suppose that $s > 0$ and $ 1 < p < \infty $.
Then there exists a constant $C > 0$ such that
\begin{equation*}
     \begin{split}
  \|\Lambda^s (fg)- f\Lambda^s g\|_{L^p(\mathbb{R}^3)}
  \leq C\left(\|\nabla f\|_{L^{p_1}(\mathbb{R}^3)}\|\Lambda^{s-1} g\|_{L^{p_2}(\mathbb{R}^3)} +
   \|\Lambda^s f\|_{L^{p_3}(\mathbb{R}^3)}\|g\|_{L^{p_4}(\mathbb{R}^3)}\right),
     \end{split}
\end{equation*}
where $ \Lambda^s:= (-\Delta)^{\frac{s}{2}} $ and $1< p_2, p_3 <\infty$ satisfying
$$\frac{1}{p}=\frac{1}{p_1}+\frac{1}{p_2}=\frac{1}{p_3}+\frac{1}{p_4}.$$
\end{lem}

In order to obtain a better description of the regularizing effect
of the transport-diffusion equation, we will use Chemin-Lerner type
spaces from \cite{Ch99, CL}.

\begin{defi}\label{chaleur+}
{\sl Let $s\in\R,$
$(r,\lambda,p)\in[1,\,+\infty]^3,$  $T\in]0,\,+\infty]$,
and $u\in{\mathcal S}'(\R^3),$ we set
$$
\|u\|_{\widetilde L^\lambda_T(B^s_{p,r})}:= \Big(2^{qs}\|\Delta_q
u\|_{L^\lambda_T(L^{p})}\Big)_{\ell ^{r}}
\quad\mbox{and}\quad
\|u\|_{\widetilde L^\lambda_T(\dot B^s_{p,r})}:= \Big(2^{qs}\|\dot\Delta_q
u\|_{L^\lambda_T(L^{p})}\Big)_{\ell ^{r}},
$$
with the usual modification if $r=\infty.$
\begin{itemize}
\item
For $s\in\R,$ we define
$\widetilde L^\lambda_T(B^s_{p,r}):= \big\{u\in{\mathcal S}'(\R^3)\;\big|\; \Vert
u\Vert_{\widetilde L^\lambda_T(B^s_{p,r})}<\infty\big\}.$

\item
For $s\le\f3p$ (resp. $s\in\R$), we define
$\widetilde{L}^{\lambda}_T(\dot B^s_{p,\,r}(\R^3))$ as the completion
of $C([0,T],\cS(\R^3))$ by norm $\|\cdot\|_{\widetilde L^\lambda_T(\dot B^s_{p,r})}.$
\end{itemize}
In the
particular case when $p=r=2,$ we denote $\widetilde L^\lambda_T(B^s_{2,2})$
(resp. $\widetilde L^\lambda_T(\dot B^s_{2,2})$) by
$\wt{L}^\la_T({H}^s)$ (resp. $\wt{L}^\la_T(\dot{H}^s).$
}
\end{defi}

\begin{rmk}\label{rmk1.222} It is easy to observe that for any $\varepsilon>0$, we have
\begin{equation*}\label{Interpo}
\Vert u\Vert_{{L}^{1}_T(H^{s-\varepsilon})} \lesssim \Vert
u\Vert_{\widetilde{L}^{1}_T( H^s)}.
\end{equation*}
Moreover, Minkowski's inequality implies that
$$
\Vert u\Vert_{\widetilde{L}^{\lambda}_T(\dot B^s_{p,r})} \leq \Vert
u\Vert_{L^{\lambda}_T(\dot B^s_{p,r})}
\quad\mbox{if}\quad\lambda\leq r \quad\hbox{and}\quad \Vert
u\Vert_{L^{\lambda}_T(\dot B^s_{p,r})} \leq \Vert
u\Vert_{\widetilde{L}^{\lambda}_T(\dot B^s_{p,r})}
\quad\mbox{if}\quad r\leq\lambda.
$$
\end{rmk}

%%%%%%%%%%%%%%%%%%%%%%%%%%%%%%%%%%%%%%%%%%%%%%%%%%%%%%%%%%%

\renewcommand{\theequation}{\thesection.\arabic{equation}}
\setcounter{equation}{0} %%%%%%%%%%%%%%%%%%%%%%%%%%

%%%%%%%%%%%%%%%%%%%%%%%%%%%%%%%%%%%%%%%%%%%%%%%%%%%%%%%%%%%
\section{Proof of Theorem \ref{thm-main}}
The main goal of this section is to give some a priori estimates and then complete the proof of Theorem \ref{thm-main}.
Let us first give the basic $L^2$ estimate for the system \eqref{MHD}.
\begin{prop}\label{prop-1}
Let $(u, B)$ be a smooth solution of the system \eqref{MHD} with $(u_0, B_0)\in L^2$. Then we have
\begin{equation*}
     \begin{split}
   \|u(t)\|_{L^2}^2+\|B(t)\|_{L^2}^2+\int_0^t\|\nabla u(\tau)\|_{L^2}^2 \, d\tau  \leq\|u_0\|_{L^2}^2+\|B_0\|_{L^2}^2.
     \end{split}
\end{equation*}
\end{prop}

\begin{proof}
Multiplying the first and second equations in \eqref{MHD} by $u$ and $B$, respectively, integrating over $\mathbb{R}^3$
and adding up, one has
\begin{equation*}
     \begin{split}
  \frac{1}{2}\frac{d}{dt}(\|u\|_{L^2}^2+\|B\|_{L^2}^2)+ \|\nabla u\|_{L^2}^2 = 0,
     \end{split}
\end{equation*}
which implies that the desired result by using Gronwall's inequality.
%\begin{equation*}
%     \begin{split}
%   \|u(t)\|_{L^2}^2+\|B(t)\|_{L^2}^2+\int_0^t\|\nabla u(\tau)\|_{L^2}^2 \, d\tau  \leq\|u_0\|_{L^2}^2+\|B_0\|_{L^2}^2.
%     \end{split}
%\end{equation*}
%The proof is complete.
\end{proof}

The next proposition describes some estimates for $\Pi$ and $\Omega$.
\begin{prop}\label{prop-2}
Let $(u, B)$ be a smooth solution of the system \eqref{MHD} with $\frac{\omega_0}{r}\in L^2$ and
$(u_0, B_0)\in H^1\times H^2$ satisfying the assumptions in Theorem \ref{thm-main}. Then there holds
\begin{equation}\label{p2-1}
     \begin{split}
   \|\Pi(t)\|_{L^p}\leq\|\Pi_0\|_{L^p}, \qquad \forall\, 2\leq p\leq 6,
     \end{split}
\end{equation}
and
\begin{equation}\label{p2-2}
     \begin{split}
   \|\Omega(t)\|_{L^2}^2+\int_0^t\|\nabla \Omega(\tau)\|_{L^2}^2 \, d\tau  \lesssim \|\Omega_0\|_{L^2}^2+\|B_0\|_{H^2}^4 t.
     \end{split}
\end{equation}
\end{prop}

\begin{proof}
Since $\Pi$ satisfies the homogeneous transport equation, the first equation in \eqref{MHD-3}, we can show \eqref{p2-1} by standard process. Taking the $L^2$ inner product of the second equation in \eqref{MHD-3} with $\Omega$, we have
\begin{equation*}
     \begin{split}
  \frac{1}{2}\frac{d}{dt}\|\Omega\|_{L^2}^2= -\int\Omega(u\cdot \nabla \Omega) \, dx +\int\Omega(\Delta+\frac{2}{r}\partial_r)\Omega \, dx-\int\Omega\partial_z \Pi^2 \, dx.
     \end{split}
\end{equation*}

Using the incompressible condition $\nabla\cdot u=0$, we obtain
\begin{equation*}
     \begin{split}
  \int\Omega(u\cdot \nabla \Omega) \, dx=0
     \end{split}
\end{equation*}
and
\begin{equation*}
     \begin{split}
  \int\Omega(\Delta+\frac{2}{r}\partial_r)\Omega \, dx=-\|\nabla \Omega\|_{L^2}^2-2\pi\int_{\mathbb{R}}|\Omega(t, 0, z)|^2\, dz.
     \end{split}
\end{equation*}
Applying integration by parts and Young's inequality, we get
\begin{equation*}
     \begin{split}
  -\int\Omega\partial_z \Pi^2 \, dx=\int \partial_z \Omega \Pi^2 \, dx\leq\|\Pi\|_{L^4}^2\|\partial_z \Omega\|_{L^2}
  \leq \frac{1}{2}\|\Pi\|_{L^4}^4 +\frac{1}{2}\|\partial_z \Omega\|_{L^2}^2.
     \end{split}
\end{equation*}

Collecting all the above estimates and \eqref{p2-1}, one has
\begin{equation}\label{p2-3}
     \begin{split}
  \frac{d}{dt}\|\Omega\|_{L^2}^2+\|\nabla \Omega\|_{L^2}^2 +4\pi\int_{\mathbb{R}}|\Omega(t, 0, z)|^2\, dz
  \leq \|\Pi_0\|_{L^4}^4.
     \end{split}
\end{equation}

Note that
%\begin{equation*}
%     \begin{split}
%  |\nabla(\nabla\times u)|^2=|\nabla \omega^\theta|^2+|\Omega|^2,
%     \end{split}
%\end{equation*}
%and
\begin{equation*}
     \begin{split}
|\nabla B|^2=|\nabla B^\theta|^2+|\Pi|^2.
     \end{split}
\end{equation*}
Therefore, we get
%\begin{equation*}
%     \begin{split}
%  \|\Omega_0\|_{L^2}\leq \|u_0\|_{H^2},
%     \end{split}
%\end{equation*}
\begin{equation*}
     \begin{split}
  \|\Pi_0\|_{L^2}\leq \|B_0\|_{H^1},\quad \mbox{and} \quad \|\Pi_0\|_{L^4}\leq \|\nabla B_0\|_{L^4}\lesssim\|B_0\|_{H^2},
     \end{split}
\end{equation*}
where we have used Sobolev embedding $H^1(\mathbb{R}^3)\hookrightarrow L^p(\mathbb{R}^3)$ $(2 \leq p \leq6)$.

Consequently, integrating \eqref{p2-3} with respect to time implies
\begin{equation*}
     \begin{split}
   \|\Omega(t)\|_{L^2}^2+\int_0^t\|\nabla \Omega(\tau)\|_{L^2}^2 \, d\tau  \leq\|\Omega_0\|_{L^2}^2+\|\Pi_0\|_{L^4}^4t
   \lesssim \|\Omega_0\|_{L^2}^2+\|B_0\|_{H^2}^4 t.
     \end{split}
\end{equation*}
This completes the proof of Proposition \ref{prop-2}.
\end{proof}

From the Biot-Savart law
$$u(x)=\frac{1}{4\pi}\int_{\mathbb{R}^3}\frac{(y-x)\wedge \omega(y)}{|y-x|^3}\, dy,$$
 we have the following lemma linking the velocity to the vorticity, which plays an important role.
\begin{lem}[\cite{Abidi2011, Lei2015}]\label{lem-1}
Let $u$ be a smooth axially symmetric vector field with zero divergence and $\omega=\omega^\theta e_\theta$ be its curl.
Then we have
\begin{equation*}
     \begin{split}
   \|u\|_{L^\infty}\lesssim \|\omega^\theta\|_{L^2}^{\frac{1}{2}}\|\nabla\omega^\theta\|_{L^2}^{\frac{1}{2}},
     \end{split}
\end{equation*}
and
\begin{equation*}
     \begin{split}
   \|\frac{u^r}{r}\|_{L^\infty}\lesssim \|\Omega\|_{L^2}^{\frac{1}{2}}\|\nabla\Omega\|_{L^2}^{\frac{1}{2}}.
     \end{split}
\end{equation*}
\end{lem}
%This lemma was proved in many literatures, such as \cite{Abidi2011} and \cite{Lei2015}, we omit the details here.

With Proposition \ref{prop-2} and Lemma \ref{lem-1} in hand, we immediately obtain the following corollary.
\begin{col}\label{col-1}
Under the assumptions of Proposition \ref{prop-2}, we have
\begin{equation*}
     \begin{split}
   \int_0^t\|\frac{u^r}{r}(\tau)\|_{L^\infty} \,d\tau \lesssim t^{\frac{5}{4}}.
     \end{split}
\end{equation*}
\end{col}
\begin{proof}
By using H\"{o}lder's inequality and Proposition \ref{prop-2}, it is easy to obtain that
\begin{equation*}
     \begin{split}
   \int_0^t\|\frac{u^r}{r}(\tau)\|_{L^\infty} \,d\tau \leq \sup_{0\leq \tau\leq t}\|\Omega(\tau, \cdot)\|_{L^2}^{\frac{1}{2}}
   \int_0^t\|\nabla \Omega(\tau)\|_{L^2}^{\frac{1}{2}} \,d\tau\lesssim t^\frac{5}{4}.
     \end{split}
\end{equation*}
\end{proof}

To be continued, we need the following key proposition.
\begin{prop}\label{prop-3}
Let $(u, B)$ be a smooth solution of the system \eqref{MHD} with $\frac{\omega_0}{r}\in L^2$ and
$(u_0, B_0)\in H^1\times H^2$ satisfying the assumptions in Theorem \ref{thm-main}. Then we have
\begin{equation*}
     \begin{split}
   \|B^\theta (t)\|_{L^p}\lesssim \|B^\theta_0\|_{L^p}\exp (Ct^{\frac{5}{4}}), \qquad \forall \, 2\leq p\leq +\infty.
     \end{split}
\end{equation*}
\end{prop}
\begin{proof}
For any $2\leq p< \infty$, multiplying the third equation in \eqref{MHD-2} by $|B^\theta|^{p-2}B^\theta$, integrating by
parts and using H\"{o}lder's inequality, one has
\begin{equation*}
     \begin{split}
   \frac{1}{p}\frac{d}{dt}\|B^\theta\|_{L^p}^p \lesssim \int |B^\theta|^p|\frac{u^r}{r}|\, dx
   \lesssim \|B^\theta\|_{L^p}^p\|\frac{u^r}{r}\|_{L^\infty},
     \end{split}
\end{equation*}
which implies
\begin{equation*}
     \begin{split}
   \frac{d}{dt}\|B^\theta\|_{L^p} \lesssim \|B^\theta\|_{L^p}\|\frac{u^r}{r}\|_{L^\infty}.
     \end{split}
\end{equation*}
Applying Gronwall's inequality and using Corollary \ref{col-1}, we get
\begin{equation}\label{p3-1}
     \begin{split}
   \|B^\theta (t)\|_{L^p} \lesssim \|B^\theta_0\|_{L^p}\exp\left(C \int_0^t\|\frac{u^r}{r}(\tau)\|_{L^\infty} \, d\tau\right)
   \lesssim \|B^\theta_0\|_{L^p}\exp (Ct^{\frac{5}{4}}).
     \end{split}
\end{equation}
Let $p\rightarrow\infty $ in \eqref{p3-1}, we complete the proof of the proposition.
\end{proof}

The following proposition describes the estimate of $\omega$.
\begin{prop}\label{prop-4}
Let $(u, B)$ be a smooth solution of the system \eqref{MHD} with $\frac{\omega_0}{r}\in L^2$ and
$(u_0, B_0)\in H^1\times H^2$ satisfying the assumptions in Theorem \ref{thm-main}. Then we have
\begin{equation*}
     \begin{split}
   \|\omega(t)\|_{L^2}^2 +\int_0^t \|\nabla\omega(\tau)\|_{L^2}^2 \, d\tau
   \lesssim \exp (Ct^{\frac{5}{4}}).
     \end{split}
\end{equation*}
\end{prop}
\begin{proof}
Recall that in cylindrical coordinates the vorticity of the swirl-free axisymmetric velocity is given by
$$\omega= \nabla \times u=\omega^\theta e_\theta$$
and satisfies
\begin{equation*}
     \begin{split}
   \partial_t \omega^\theta +u\cdot\nabla \omega^\theta-(\Delta-\frac{1}{r^2})\omega^\theta-\frac{u^r}{r}\omega^\theta
   =-\partial_z\frac{(B^\theta)^2}{r}.
     \end{split}
\end{equation*}

Taking the $L^2$ inner product of $\omega^\theta$ equation with $\omega^\theta$ and using the incompressible condition $\nabla\cdot u=0$,
we get
\begin{equation}\label{p4-1}
     \begin{split}
   \frac{1}{2}\frac{d}{dt}\|\omega^\theta\|_{L^2}^2 +\|\nabla\omega^\theta\|_{L^2}^2+\|\frac{\omega^\theta}{r}\|_{L^2}^2
   &\leq \int \frac{u^r}{r}|\omega^\theta|^2 \, dx - \int \partial_z\frac{(B^\theta)^2}{r}\omega^\theta \, dx\\
   &:= I_1 + I_2.
     \end{split}
\end{equation}
For $I_1$, one has
\begin{equation*}\label{p4-2}
     \begin{split}
   |I_1|\leq \|\frac{u^r}{r}\|_{L^\infty}\|\omega^\theta\|_{L^2}^2.
     \end{split}
\end{equation*}
For $I_2$, it follows from integration by parts that
\begin{equation*}\label{p4-3}
     \begin{split}
   |I_2|=\left|\int \frac{(B^\theta)^2}{r} \partial_z \omega^\theta \, dx \right|
   \leq \|B^\theta\|_{L^\infty}\|\frac{B^\theta}{r}\|_{L^2}\|\partial_z \omega^\theta\|_{L^2}
   \leq \frac{1}{2}\|B^\theta\|_{L^\infty}^2\|\Pi\|_{L^2}^2+ \frac{1}{2}\|\partial_z \omega^\theta\|_{L^2}^2.
     \end{split}
\end{equation*}

Inserting the above estimates into \eqref{p4-1} and using Proposition \ref{prop-3}, it infers
\begin{equation*}
     \begin{split}
   \frac{d}{dt}\|\omega^\theta\|_{L^2}^2 +\|\nabla\omega^\theta\|_{L^2}^2+\|\frac{\omega^\theta}{r}\|_{L^2}^2
   &\lesssim \|\frac{u^r}{r}\|_{L^\infty}\|\omega^\theta\|_{L^2}^2 + \|B^\theta\|_{L^\infty}^2\|\Pi\|_{L^2}^2\\
   &\lesssim \|\frac{u^r}{r}\|_{L^\infty}\|\omega^\theta\|_{L^2}^2 + \|B^\theta_0\|_{H^2}^4 \exp (Ct^{\frac{5}{4}}),
     \end{split}
\end{equation*}
where we have used the Sobolev embedding $H^m(\mathbb{R}^3)\hookrightarrow L^\infty(\mathbb{R}^3)$ for $m>\frac{3}{2}$.

Hence, the Gronwall inequality and Corollary \ref{col-1} ensure that
\begin{equation*}
     \begin{split}
   \|\omega^\theta(t)\|_{L^2}^2 &+\int_0^t \|\nabla\omega^\theta(\tau)\|_{L^2}^2 \, d\tau +
   \int_0^t \|\frac{\omega^\theta}{r}(\tau)\|_{L^2}^2 \, d\tau\\
   &\,\lesssim \left( \|\omega_0^\theta\|_{L^2}^2+ \|B^\theta_0\|_{H^2}^4\int_0^t \exp (C\tau^\frac{5}{4})\, d\tau      \right) \exp \left(C \int_0^t \|\frac{u^r}{r}(\tau)\|_{L^\infty} \, d\tau\right)\\
   &\,\lesssim (1+t)\exp (Ct^{\frac{5}{4}}) \lesssim \exp (Ct^{\frac{5}{4}}).
     \end{split}
\end{equation*}

Noting
\begin{equation*}
     \begin{split}
   \|\omega\|_{L^2}= \|\omega^\theta\|_{L^2},\quad \mbox{and} \quad  \|\nabla \omega\|_{L^2}^2= \|\nabla \omega^\theta\|_{L^2}^2 + \|\frac{\omega^\theta}{r}\|_{L^2}^2,
     \end{split}
\end{equation*}
%and
%\begin{equation*}
%     \begin{split}
%   \|\nabla \omega\|_{L^2}^2= \|\nabla \omega^\theta\|_{L^2}^2 + \|\frac{\omega^\theta}{r}\|_{L^2}^2,
%     \end{split}
%\end{equation*}
we get
\begin{equation*}
     \begin{split}
   \|\omega(t)\|_{L^2}^2 +\int_0^t \|\nabla\omega(\tau)\|_{L^2}^2 \, d\tau
  \lesssim \exp (Ct^{\frac{5}{4}}).
     \end{split}
\end{equation*}
This completes the proof of Proposition \ref{prop-4}.
\end{proof}

Consequently, we have the following corollary.
\begin{col}\label{col-2}
Under the assumptions of Proposition \ref{prop-4}, we have
\begin{equation}\label{c2-1}
     \begin{split}
   \|\nabla u(t)\|_{L^2}^2+ \int_0^t \|\nabla^2 u(\tau)\|_{L^2}^2 \, d\tau \lesssim \exp (Ct^{\frac{5}{4}}),
     \end{split}
\end{equation}
and
\begin{equation}\label{c2-2}
     \begin{split}
\int_0^t \|u(\tau)\|_{L^\infty}^2 \, d\tau \lesssim \exp (Ct^{\frac{5}{4}}).
     \end{split}
\end{equation}
\end{col}
\begin{proof}
By virtue of the vector identity $\nabla\times\nabla\times u=-\Delta u+\nabla\nabla\cdot u$ and $\nabla\cdot u=0$,
we see
\begin{equation*}
     \begin{split}
   \nabla u=\nabla(-\Delta)^{-1}\nabla\times \omega.
     \end{split}
\end{equation*}
Using the Calder$\operatorname{\acute{o}}$n-Zygmund inequality yields
\begin{equation}\label{p4-5}
     \begin{split}
   \|\nabla u(t)\|_{L^p}\leq C(p)\|\omega(t)\|_{L^p}, \qquad \forall \, 1<p<+\infty.
     \end{split}
\end{equation}

In particular, taking $p=2$ in \eqref{p4-5} and combining Proposition \ref{prop-4} lead to
the desired \eqref{c2-1}. Using Lemma \ref{lem-1} and Proposition \ref{prop-4}, we have
\begin{equation*}
     \begin{split}
   \int_0^t \|u(\tau)\|_{L^\infty}^2 \, d\tau
   &\lesssim  \int_0^t \|\omega^\theta(\tau)\|_{L^2}\|\nabla\omega^\theta(\tau)\|_{L^2} \, d\tau\\
   &\lesssim  \sup_{0\leq \tau\leq t}\|\omega^\theta(\tau)\|_{L^2} \left(\int_0^t \|\nabla\omega^\theta(\tau)\|_{L^2}^2 \, d\tau\right)^{\frac{1}{2}}(\int_0^t 1 \, d\tau)^{\frac{1}{2}}\\
   &\lesssim \exp (Ct^{\frac{5}{4}}),
     \end{split}
\end{equation*}
which gives the desired \eqref{c2-2}. This ends the proof of Corollary \ref{col-2}.
\end{proof}

Now, let us derive the $L^1([0, T]; \operatorname{Lip}(\mathbb{R}^3))$ estimate for $u$.

\begin{prop}\label{prop-5}
Let $(u, B)$ be a smooth solution of the system \eqref{MHD} with $\frac{\omega_0}{r}\in L^2$ and
$(u_0, B_0)\in H^1\times H^2$ satisfying the assumptions in Theorem \ref{thm-main}. Then for every $3<p\leq 6$
\begin{equation*}
     \begin{split}
  \int_0^t \|u(\tau)\|_{B^{1+\frac{3}{p}}_{p, 1}} \, d\tau + \int_0^t \|\nabla u(\tau)\|_{L^\infty} \, d\tau\lesssim \exp (Ct^{\frac{5}{4}}).
     \end{split}
\end{equation*}
\end{prop}
\begin{proof}
Rewriting the equation for vorticity $\omega =\nabla \times u$, one has
\begin{equation*}
     \begin{split}
   \partial_t \omega -\Delta\omega=-\nabla\times(u\cdot \nabla u- B\cdot \nabla B).
     \end{split}
\end{equation*}
Using the vector identity
\begin{equation*}
  \begin{split}
(\nabla\times f)\times f =-\frac{1}{2}\nabla |f|^2+ f\cdot \nabla f,
  \end{split}
\end{equation*}
we obtain
\begin{equation*}
  \begin{split}
\nabla\times (f\cdot \nabla f) = \nabla\times \left((\nabla\times f)\times f \right).
  \end{split}
\end{equation*}
A routine computation gives rise to
$$\nabla \times \left((\nabla\times B)\times B\right)=-\partial_z(\Pi B^\theta e_\theta).$$
Thus,
\begin{equation}\label{p5-1}
     \begin{split}
   \partial_t \omega -\Delta\omega=-\nabla\times (u\cdot \nabla u) - \partial_z(\Pi B^\theta e_\theta).
     \end{split}
\end{equation}

Let $q\in \mathbb{N}$ and $\omega_q :=\Delta_q \omega$. Then localizing in frequency to the vorticity equation \eqref{p5-1} and applying Duhamel formula, we know
\begin{equation*}
     \begin{split}
   \omega_q=e^{t\Delta}\omega_q(0) - \int_0^t e^{{(t-\tau)}\Delta}\Delta_q \left(\nabla\times(u\cdot \nabla u)\right)(\tau) \, d\tau - \int_0^t e^{{(t-\tau)}\Delta}\Delta_q \left(\partial_z(\Pi B^\theta e_\theta)\right)(\tau) \, d\tau.
     \end{split}
\end{equation*}

Thanks to the estimate, see \cite{Che1998},
\begin{equation*}
     \begin{split}
   \|e^{t\Delta}\Delta_q f\|_{L^m}\leq Ce^{-ct2^{2q}} \|\Delta_q f\|_{L^m}, \quad \forall\, 1 \leq m\leq \infty,
     \end{split}
\end{equation*}
and using Bernstein inequality, we get
\begin{equation*}
     \begin{split}
   \|\omega_q\|_{L^p}\lesssim e^{-ct 2^{2q}}\|\omega_q(0)\|_{L^p}
   &+2^{2q} \int_0^t e^{-c(t-\tau)2^{2q}}\|\Delta_q (u\otimes u)(\tau)\|_{L^p}  \, d\tau\\
   &+2^q \int_0^t e^{-c(t-\tau)2^{2q}} \|\Delta_q(\Pi B^\theta)(\tau)\|_{L^p} \, d\tau.
     \end{split}
\end{equation*}
Then integrating in time and using convolution inequalities, one has
\begin{equation*}
     \begin{split}
   \int_0^t \|\omega_q(\tau)\|_{L^p} \, d\tau\lesssim 2^{-2q}\|\omega_q(0)\|_{L^p}
   +\int_0^t \|\Delta_q (u\otimes u)(\tau)\|_{L^p}  \, d\tau
   +2^{-q} \int_0^t \|\Delta_q(\Pi B^\theta)(\tau)\|_{L^p} \, d\tau,
     \end{split}
\end{equation*}
which implies that
\begin{equation*}
     \begin{split}
   \int_0^t \|\omega(\tau)\|_{B^{\frac{3}{p}}_{p, 1}} \, d\tau
   &\lesssim \int_0^t \|\Delta_{-1}\omega(\tau)\|_{L^p} \, d\tau
   +\|\omega_0\|_{B^{\frac{3}{p}-2}_{p, 1}}\\
   &\quad +\int_0^t \|(u\otimes u)(\tau)\|_{B^{\frac{3}{p}}_{p, 1}}  \, d\tau
   +\int_0^t \|(\Pi B^\theta)(\tau)\|_{B^{\frac{3}{p}-1}_{p, 1}} \, d\tau.
     \end{split}
\end{equation*}

We take $3<p\leq 6$. For the first term of the r.h.s, we get from Bernstein inequality and Proposition \ref{prop-4} that
\begin{equation*}
     \begin{split}
   \int_0^t \|\Delta_{-1}\omega(\tau)\|_{L^p} \, d\tau \lesssim t \|\omega\|_{L^\infty([0, t]; L^2(\mathbb{R}^3))}\lesssim \exp (Ct^\frac{5}{4}).
     \end{split}
\end{equation*}
For the second term of the r.h.s, using Besov embedding implies
\begin{equation*}
     \begin{split}
   \|\omega_0\|_{B^{\frac{3}{p}-2}_{p, 1}}\lesssim \|u_0\|_{B^{\frac{3}{p}-1}_{p, 1}}
   \lesssim \|u_0\|_{B^{\frac{1}{2}}_{2, 1}} \lesssim \|u_0\|_{H^1}.
     \end{split}
\end{equation*}
Applying Besov embedding, law products and interpolation inequality, we have
\begin{equation*}
     \begin{split}
    \|u\otimes u\|_{B^{\frac{3}{p}}_{p, 1}}&\lesssim \|u\otimes u\|_{B^{\frac{3}{2}}_{2, 1}}
    \lesssim \|u\|_{L^\infty}\|u\|_{B^{\frac{3}{2}}_{2, 1}}\\
    &\lesssim \|u\|_{L^\infty}\|u\|_{L^2}+ \|u\|_{L^\infty}\|\nabla u\|_{B^{\frac{1}{2}}_{2, 1}}\\
    &\lesssim \|u\|_{L^\infty}\|u\|_{L^2}+ \|u\|_{L^\infty}\|\nabla u\|_{L^2}^{\frac{1}{2}}\|\nabla^2 u\|_{L^2}^{\frac{1}{2}},
     \end{split}
\end{equation*}
which together with Corollary \ref{col-2} implies
\begin{equation*}
     \begin{split}
    &\|u \otimes u\|_{L^1([0, t]; B^{\frac{3}{p}}_{p, 1}(\mathbb{R}^3))}\\
    &\lesssim t^\frac{1}{2}\|u\|_{L^2([0, t]; L^\infty(\mathbb{R}^3))}\|u\|_{L^\infty([0, t]; L^2(\mathbb{R}^3))}\\
    &\quad + \|u\|_{L^\frac{4}{3}([0, t]; L^\infty(\mathbb{R}^3))}\|\nabla u\|_{L^\infty([0, t]; L^2(\mathbb{R}^3))}^{\frac{1}{2}}\|\nabla^2 u\|_{L^2([0, t]; L^2(\mathbb{R}^3))}^{\frac{1}{2}}\\
    &\lesssim t^\frac{1}{2}\|u\|_{L^2([0, t]; L^\infty(\mathbb{R}^3))}\|u\|_{L^\infty([0, t]; L^2(\mathbb{R}^3))}\\
    &\quad + t^{\frac{1}{4}}\|u\|_{L^2([0, t]; L^\infty(\mathbb{R}^3))}\|\nabla u\|_{L^\infty([0, t]; L^2(\mathbb{R}^3))}^{\frac{1}{2}}\|\nabla^2 u\|_{L^2([0, t]; L^2(\mathbb{R}^3))}^{\frac{1}{2}}\\
    &\lesssim \exp (Ct^{\frac{5}{4}}).
     \end{split}
\end{equation*}
We use the embedding $L^p\hookrightarrow B^{\frac{3}{p}-1}_{p, 1}$ for $p >3$,
\begin{equation*}
     \begin{split}
   \|\Pi B^\theta\|_{B^{\frac{3}{p}-1}_{p, 1}} \lesssim \|\Pi B^\theta\|_{L^p}\lesssim \|B^\theta\|_{L^\infty}\|\Pi\|_{L^p},
     \end{split}
\end{equation*}
which gives for $3<p\leq6 $
\begin{equation*}
     \begin{split}
   \int_0^t \|\Pi B^\theta(\tau)\|_{B^{\frac{3}{p}-1}_{p, 1}} \, d\tau \lesssim \|\Pi_0\|_{L^p}\int_0^t \|B^\theta(\tau)\|_{ L^\infty}\, d\tau
   \lesssim t\exp (Ct^{\frac{5}{4}}).
     \end{split}
\end{equation*}

Hence, we have
\begin{equation*}
     \begin{split}
  \int_0^t \|\omega(\tau)\|_{B^{\frac{3}{p}}_{p, 1}} \, d\tau\lesssim \exp (Ct^{\frac{5}{4}}).
     \end{split}
\end{equation*}
And then using the Besov embedding $B^{\frac{3}{p}+1}_{p, 1}\hookrightarrow W^{1, \infty}$ implies
\begin{equation*}
     \begin{split}
  \int_0^t \|\nabla u(\tau)\|_{L^\infty}\, d\tau \lesssim \int_0^t \|u(\tau)\|_{B^{\frac{3}{p}+1}_{p, 1}} \, d\tau
  \lesssim \int_0^t\|\omega(\tau)\|_{B^{\frac{3}{p}}_{p, 1}}\, d\tau \lesssim \exp (Ct^{\frac{5}{4}}).
     \end{split}
\end{equation*}
The concludes the proof.
\end{proof}

We give the following crucial proposition for $\nabla B$.
\begin{prop}\label{prop-7}
Let $(u, B)$ be a smooth solution of the system \eqref{MHD} with $\frac{\omega_0}{r}\in L^2$ and
$(u_0, B_0)\in H^1\times H^2$ satisfying the assumptions in Theorem \ref{thm-main}. Then there holds
\begin{equation*}
     \begin{split}
  \|\nabla B (t)\|_{L^p}\lesssim \exp\left(C \exp (Ct^{\frac{5}{4}})\right),
  \qquad \forall \, 2\leq p \leq 6.
     \end{split}
\end{equation*}
\end{prop}
\begin{proof}
We first write the second equation in \eqref{MHD} as
\begin{equation}\label{p7-1}
     \begin{split}
  \partial_t B + u\cdot \nabla B =\frac{u^r}{r}B.
     \end{split}
\end{equation}
Applying the operator $\nabla$ to \eqref{p7-1}, it infers
\begin{equation*}
     \begin{split}
  \partial_t \nabla B + \nabla u\cdot \nabla B +u\cdot \nabla\nabla B
  -\frac{u^r}{r} \nabla B -\nabla u^r\frac{B}{r} e_\theta -(\nabla \frac{1}{r}) u^r B =0.
     \end{split}
\end{equation*}
A direct computation gives
\begin{equation*}
     \begin{split}
   (\nabla \frac{1}{r}) u^r B =-\frac{1}{r^2}e_r u^r B= -\frac{B^\theta}{r^2} u^r e_r\otimes e_\theta.
     \end{split}
\end{equation*}
This yields to
\begin{equation}\label{p7-2}
     \begin{split}
  \partial_t \nabla B + \nabla u\cdot \nabla B +u\cdot \nabla\nabla B
  -\frac{u^r}{r} \nabla B -\nabla u^r\frac{B}{r} e_\theta +\frac{u^r}{r} \Pi e_r \otimes e_\theta =0.
     \end{split}
\end{equation}

For $2\leq p \leq6$, multiplying the equation \eqref{p7-2} by $|\nabla B|^{p-2}\nabla B$ and integrating by parts, we deduce
\begin{equation*}
     \begin{split}
  \frac{1}{p}\frac{d}{dt}\|\nabla B\|_{L^p}^p
  \leq \left(\|\nabla u\|_{L^\infty}+ \|\frac{u^r}{r}\|_{L^\infty}\right)\|\nabla B\|_{L^p}^p + \left(\|\nabla u^r\|_{L^\infty}+ \|\frac{u^r}{r}\|_{L^\infty}\right)\|\Pi\|_{L^p}\|\nabla B\|_{L^{p}}^{p-1}.
     \end{split}
\end{equation*}
Thus,
\begin{equation*}
     \begin{split}
  \frac{d}{dt}\|\nabla B\|_{L^p} \leq \left(\|\nabla u\|_{L^\infty}+ \|\frac{u^r}{r}\|_{L^\infty}\right)\|\nabla B\|_{L^p} + \left(\|\nabla u\|_{L^\infty}+ \|\frac{u^r}{r}\|_{L^\infty}\right)\|\Pi\|_{L^p}.
     \end{split}
\end{equation*}
Applying Gronwall's inequality implies
\begin{equation*}
     \begin{split}
  \|\nabla B(t)\|_{L^p}
  &\leq \left( \|\nabla B_0\|_{L^p} +
  \|\Pi_0\|_{L^p}\int_0^t (\|\nabla u(\tau)\|_{L^\infty}+ \|\frac{u^r}{r}(\tau)\|_{L^\infty})\, d\tau  \right)\\
  &\quad\,  \times \exp \int_0^t \left(\|\nabla u(\tau)\|_{L^\infty}+ \|\frac{u^r}{r}(\tau)\|_{L^\infty}\right)\, d\tau\\
  &\lesssim \exp\left(C \exp (Ct^{\frac{5}{4}}) \right),
     \end{split}
\end{equation*}
where we have used Proposition \ref{prop-2}, Corollary \ref{col-1}, Proposition \ref{prop-5} and the Sobolev embedding
$H^1(\mathbb{R}^3)\hookrightarrow L^6(\mathbb{R}^3)$. This achieves the proof of Proposition \ref{prop-7}.
\end{proof}

Next performing the prior $H^2$ estimate for $B$, we first show the following two propositions.
\begin{prop}\label{prop-7a}
Let $(u, B)$ be a smooth solution of the system \eqref{MHD} with $\frac{\omega_0}{r}\in L^2$ and
$(u_0, B_0)\in H^1\times H^2$ satisfying the assumptions in Theorem \ref{thm-main}. Then there holds
\begin{equation*}
     \begin{split}
  \|\nabla \Pi (t)\|_{L^2}\lesssim \exp\left(C\exp (Ct^{\frac{5}{4}})\right).
     \end{split}
\end{equation*}
\end{prop}
\begin{proof}
Applying the operator $\nabla$ to the equation of $\Pi$ in \eqref{MHD-3}, one has
\begin{equation*}
     \begin{split}
  \partial_t \nabla \Pi +\nabla u\cdot \nabla \Pi + u\cdot \nabla\nabla \Pi = 0.
     \end{split}
\end{equation*}

Taking the $L^2$ inner product with $\nabla \Pi$, we obtain from H\"{o}lder's inequality that
\begin{equation*}
     \begin{split}
  \frac{1}{2}\frac{d}{dt} \|\nabla \Pi\|_{L^2}^2
  =- \int \nabla u\cdot \nabla \Pi\cdot \nabla \Pi\, dx \leq \|\nabla u\|_{L^\infty} \|\nabla \Pi\|_{L^2}^2.
     \end{split}
\end{equation*}
%\begin{equation*}
%     \begin{split}
%  \left| \int \nabla u\cdot \nabla \Pi\cdot \nabla \Pi\, dx \right|
%  \leq \|\nabla u\|_{L^\infty} \|\nabla \Pi\|_{L^2}^2.
%     \end{split}
%\end{equation*}
Thus, using Gronwall's inequality implies
\begin{equation*}
     \begin{split}
  \|\nabla \Pi (t)\|_{L^2}\leq \|\nabla \Pi_0\|_{L^2} \exp \int_0^t \|\nabla u(\tau)\|_{L^\infty}\, d\tau
  \lesssim \|B_0\|_{H^2} \exp \left(C\exp (Ct^{\frac{5}{4}})\right),
     \end{split}
\end{equation*}
and then the proof of Proposition \ref{prop-7a} is completed.
\end{proof}

Before proving next proposition, we recall the following estimate for the heat equation
\begin{equation}\label{heat-eq}
    \begin{cases}
    \partial_t f - \Delta f = F, \quad (t, x)\in \mathbb{R}^+\times \mathbb{R}^3, \\
    f|_{t=0}=f_0.
    \end{cases}
\end{equation}
\begin{lem}[\cite{Dan2005}]\label{lem-2}
Let $t>0$, $s\in \mathbb{R}$ and $1\leq \rho, p, r \leq \infty$. Assume that $f_0\in \dot{B}^{s}_{p, r}$ and
$F\in \widetilde{L}^\rho([0, t]; \dot{B}^{s-2+\frac{2}{\rho}}_{p, r})$. Then the equation \eqref{heat-eq}
has a unique solution $f\in \widetilde{L}^\rho([0, t]; \dot{B}^{s+\frac{2}{\rho}}_{p, r})\cap
\widetilde{L}^\infty([0, t]; \dot{B}^{s}_{p, r})$
and the following estimate
holds for all $\rho_1 \in [\rho, +\infty]$,
\begin{equation*}
     \begin{split}
  \|f\|_{\widetilde{L}^{\rho_1}([0, t]; \dot{B}^{s+\frac{2}{\rho_1}}_{p, r})}\leq C \left( \|f_0\|_{\dot{B}^s_{p, r}}
  + \|F\|_{\widetilde{L}^{\rho}([0, t]; \dot{B}^{s-2+\frac{2}{\rho}}_{p, r})} \right).
     \end{split}
\end{equation*}
\end{lem}

\begin{prop}\label{prop-7b}
Let $\epsilon \in (0, 1)$, $(u, B)$ be a smooth solution of the system \eqref{MHD} with $\frac{\omega_0}{r}\in L^2$ and
$(u_0, B_0)\in H^1\times H^2$ satisfying the assumptions in Theorem \ref{thm-main}. Then one has
\begin{equation*}
     \begin{split}
  \int_0^t\|\nabla u(\tau)\|_{{H}^{2-\epsilon}}\, d\tau
  \lesssim \exp\left(C \exp (Ct^{\frac{5}{4}}) \right).
     \end{split}
\end{equation*}
\end{prop}
\begin{proof}
Note that the equation of vorticity
\begin{equation*}
     \begin{split}
   \partial_t \omega -\Delta\omega=-\nabla\times (u\cdot \nabla u) - \partial_z(\Pi B^\theta e_\theta),
     \end{split}
\end{equation*}
we obtain from Lemma \ref{lem-2} and Remark \ref{rmk1.222} that
\begin{equation*}
     \begin{split}
   \|\omega\|_{\widetilde{L}^1([0, t]; \dot{H}^2(\mathbb{R}^3))}\lesssim \|\omega_0\|_{L^2} + \|\nabla\times (u\cdot \nabla u)\|_{{L}^1([0, t]; L^2(\mathbb{R}^3))} + \|\partial_z(\Pi B^\theta )\|_{{L}^1([0, t]; L^2(\mathbb{R}^3))}.
   %&\, \leq  \|u\cdot \nabla u\|_{L^1([0, t]; H^1(\mathbb{R}^3))} + \|\Pi B^\theta\|_{L^1([0, t]; H^1(\mathbb{R}^3))}.
     \end{split}
\end{equation*}
Using Lemma \ref{lem-Besov}, one has
\begin{equation*}
     \begin{split}
  \|\nabla\times (u\cdot \nabla u)\|_{L^2(\mathbb{R}^3)}\leq \|u\cdot \nabla u\|_{H^1}\lesssim \|u\|_{L^\infty}\|\nabla u\|_{H^1} +\|u\|_{H^1}\|\nabla u\|_{L^\infty}
     \end{split}
\end{equation*}
and
\begin{equation*}
     \begin{split}
 \|\partial_z(\Pi B^\theta )\|_{L^2(\mathbb{R}^3)}&\leq \|B^\theta \partial_z \Pi \|_{L^2(\mathbb{R}^3)} +\|\Pi \partial_z B^\theta \|_{L^2(\mathbb{R}^3)}\\
 &\leq \|B^\theta \|_{L^\infty(\mathbb{R}^3)}\|\nabla \Pi \|_{L^2(\mathbb{R}^3)}+ \|\Pi\|_{L^4(\mathbb{R}^3)}\|\nabla B^\theta \|_{L^4(\mathbb{R}^3)}.
     \end{split}
\end{equation*}
Thus, we get
\begin{equation*}
     \begin{split}
   &\|\omega\|_{\widetilde{L}^1([0, t];\dot{H}^2(\mathbb{R}^3))}\lesssim \|\omega_0\|_{L^2} + \|u\|_{L^2([0, t]; L^\infty(\mathbb{R}^3))}\|\nabla u\|_{L^2([0, t]; H^1(\mathbb{R}^3))} \\
   &\quad +
   \|u\|_{L^\infty([0, t]; H^1(\mathbb{R}^3))}\|\nabla u\|_{L^1([0, t]; L^\infty(\mathbb{R}^3))}+ t\,\|B^\theta\|_{L^\infty([0, t]; L^\infty(\mathbb{R}^3))}\|\nabla \Pi\|_{L^\infty([0, t]; L^2(\mathbb{R}^3))}
   \\
   &\qquad\quad+ t \,\|\Pi_0\|_{L^4}\|\nabla B^\theta \|_{L^\infty([0, t]; L^4(\mathbb{R}^3))} \lesssim \exp\left(C \exp (Ct^{\frac{5}{4}}) \right).
     \end{split}
\end{equation*}

Thanks to Remark \ref{rmk1.222} and Proposition \ref{prop-4}, one can see that
\begin{equation*}
     \begin{split}
   &\|\omega\|_{{L}^1([0, t]; {H}^{2-\epsilon}(\mathbb{R}^3))}\lesssim \|\omega\|_{\widetilde{L}^1([0, t];{H}^2(\mathbb{R}^3))}\lesssim \int_0^t\|\omega(\tau)\|_{L^2}\,d\tau+ \|\omega\|_{\widetilde{L}^1([0, t];\dot{H}^2(\mathbb{R}^3))}\\
   &\, \lesssim \exp (Ct^{\frac{5}{4}})+\exp\left(C \exp (Ct^{\frac{5}{4}}) \right)\lesssim\exp\left(C \exp (Ct^{\frac{5}{4}}) \right),
     \end{split}
\end{equation*}
which gives
\begin{equation*}
     \begin{split}
   \int_0^t\|\nabla u(\tau)\|_{{H}^{2-\epsilon}}\, d\tau \lesssim\int_0^t \|\omega(\tau) \|_{{H}^{2-\epsilon}}\, d\tau
 \lesssim   \exp \left(C \exp (C t^{\frac{5}{4}})  \right) .
     \end{split}
\end{equation*}
This ends the proof of the proposition.
\end{proof}

To the end, we give the $H^2$ estimate of $B$.
\begin{prop}\label{prop-8}
Let $(u, B)$ be a smooth solution of the system \eqref{MHD} with $\frac{\omega_0}{r}\in L^2$ and
$(u_0, B_0)\in H^1\times H^2$ satisfying the assumptions in Theorem \ref{thm-main}. Then one has, for any $t>0$,
\begin{equation*}
     \begin{split}
\|\nabla^2 B(t)\|_{L^2}^2\lesssim  \exp \left\{C\exp \left(C \exp ( C t^{\frac{5}{4}})  \right)\right\}.
     \end{split}
\end{equation*}
\end{prop}
\begin{proof}
Applying the operator $\nabla^2$ to the equation of $B^{\theta}$ in \eqref{MHD-2} leads to
\begin{equation*}
    \begin{split}
   & \partial_t \nabla^2 B^\theta + u^r\partial_r \nabla^2B^\theta+ u^z\partial_z \nabla^2B^\theta=(\frac{u^r }{r}\nabla^2 B^\theta+2\nabla\frac{u^r }{r}\nabla B^\theta+\Pi\,r\nabla^2\frac{u^r }{r})\\
   &
    \qquad\qquad\qquad-(\nabla^2u^r\partial_r  B^\theta+ \nabla^2u^z\partial_z B^\theta+2\nabla u^r\partial_r \nabla B^\theta+ 2\nabla u^z\partial_z \nabla B^\theta).
   \end{split}
\end{equation*}

Taking the $L^2$ inner product with $\nabla^2 B^\theta $, we obtain from the incompressible condition $\nabla\cdot u=0$ that
\begin{equation*}
     \begin{split}
   &\frac{1}{2}\frac{d}{dt} \|\nabla^2 B^{\theta}\|_{L^2}^2=\int (\frac{u^r }{r}\nabla^2 B^\theta+2\nabla\frac{u^r }{r}\nabla B^\theta+\Pi\,r\nabla^2\frac{u^r }{r})\cdot\nabla^2  B^{\theta}\, dx\\
   &\quad -\int(\nabla^2u^r\partial_r  B^\theta+ \nabla^2u^z\partial_z B^\theta+2\nabla u^r\partial_r \nabla B^\theta+ 2\nabla u^z\partial_z \nabla B^\theta)\cdot \nabla^2 B^{\theta}\, dx\\
   &:=J_1+J_2.
     \end{split}
\end{equation*}

In the following, we estimate $J_i$ term by term. For $J_1$, thanks to H\"{o}lder's inequality and the Sobolev inequality, we use  to get
\begin{equation*}
     \begin{split}
     |J_1|&\lesssim \|\frac{u^r }{r}\|_{L^\infty}\|\nabla^2 B^{\theta}\|_{L^2}^2+\|\nabla\frac{u^r }{r}\|_{L^3}\|\nabla B^\theta\|_{L^6}\|\nabla^2 B^{\theta}\|_{L^2}+\|\Pi\|_{L^6}\|r\nabla^2\frac{u^r }{r}\|_{L^3}\|\nabla^2 B^{\theta}\|_{L^2}\\
     &\lesssim  (\|\nabla u\|_{L^\infty} +\|\nabla^2 u\|_{\dot{H}^{\frac{1}{2}}})  \|\nabla^2 B^{\theta}\|_{L^2}^2+\|\nabla^2 u\|_{\dot{H}^{\frac{1}{2}}} \|\Pi\|_{L^6}^2.
     \end{split}
\end{equation*}
Thanks to  H\"{o}lder's inequality and the Sobolev inequality $\|f\|_{L^6}\lesssim  \|\nabla f\|_{L^2} $, we obtain
\begin{equation*}
     \begin{split}
     |J_2|&\lesssim\|\nabla^2 u\|_{L^3} \|\nabla B^{\theta}\|_{L^6}\|\nabla^2 B^{\theta}\|_{L^2}+\|\nabla u\|_{L^\infty} \|\nabla^2 B^{\theta}\|_{L^2}^2\\
     &\lesssim  (\|\nabla u\|_{L^\infty} +\|\nabla^2 u\|_{\dot{H}^{\frac{1}{2}}})  \|\nabla^2 B^{\theta}\|_{L^2}^2.
     \end{split}
\end{equation*}
Putting together the above estimates, we get
\begin{equation*}
     \begin{split}
\frac{d}{dt}\|\nabla^2 B^{\theta}\|_{L^2}^2\lesssim(\|\nabla u\|_{L^\infty}+\|\nabla u\|_{H^{\frac{3}{2}}})\|\nabla^2 B^{\theta}\|_{L^2}^2+\|\nabla u\|_{H^{\frac{3}{2}}}\|\Pi\|_{L^6}^2,
     \end{split}
\end{equation*}
and combining Proposition \ref{prop-5},  Proposition \ref{prop-7b} and Gronwall's inequality, we deduce
\begin{equation*}
     \begin{split}
  \|\nabla^2 B^{\theta}(t)\|_{L^2}^2 &\lesssim \left(\|\nabla^2 B^{\theta}_0\|_{L^2}^2+\|\Pi_0\|_{L^6}^2\int_0^t\|\nabla u(\tau)\|_{H^{\frac{3}{2}}}\,d\tau\right)\\
  &\quad\times\exp \left( C\int_0^t (\|\nabla u(\tau)\|_{L^\infty} +\|\nabla u(\tau)\|_{H^{\frac{3}{2}}}) \, d\tau \right)
  \\
  &\lesssim  \exp \left\{C\exp \left(C \exp (C t^{\frac{5}{4}})  \right)\right\}.
     \end{split}
\end{equation*}
This completes the proof of Proposition \ref{prop-8}.
\end{proof}

\begin{proof}[Proof of Theorem \ref{thm-main}]
With the Corollary \ref{col-2}, Proposition \ref{prop-5}, Proposition \ref{prop-7b} and Proposition \ref{prop-8}, by taking advantage of the local existence and uniqueness result, that is, Lemma \ref{lem-0}, we complete the proof of Theorem \ref{thm-main}.
\end{proof}

\noindent {\bf Acknowledgments.} The authors would like to thank Professor Guilong Gui for his guidance on this project. We would also like to thank the referees for their constructive suggestions
and comments. The work is partially supported by the National Natural Science Foundation of China under the grants 11571279, 11601423 and 11931013.\\

\end{document}